\documentclass{article}

\usepackage{amsmath}
\usepackage{amsthm}
\usepackage{alltt}
\usepackage{hyperref}
\usepackage{algorithm}
\usepackage{algpseudocode}
\usepackage{graphicx}

\numberwithin{equation}{section}

\newtheorem{theorem}{Theorem}[section]

\newtheorem{definition}{Definition}[section]

\newcommand{\gaussian}[2]
{\genfrac{(}{)}{0pt}{}{#1}{#2}_{\textstyle q}}
\newcommand{\gaussiansqr}[2]
{\genfrac{(}{)}{0pt}{}{#1}{#2}_{\textstyle q^2}}
\newcommand{\gaussiancube}[2]
{\genfrac{(}{)}{0pt}{}{#1}{#2}_{\textstyle q^3}}
\newcommand{\gaussianquart}[2]
{\genfrac{(}{)}{0pt}{}{#1}{#2}_{\textstyle q^4}}
\newcommand{\gaussianpow}[3]
{\genfrac{(}{)}{0pt}{}{#1}{#2}_{\textstyle q^#3}}
\newcommand{\dblsum}[3]
{ \mathop{\sum\sum}_{
   \genfrac{}{}{0pt}{}{\scriptstyle #1=0\ #2=0}
     {\scriptstyle #3}}^{\infty\ \ \infty}
}
\newcommand{\dblsumeven}
{ \mathop{\sum_{k=0}^n\sum_{l=0}^{n-k}}_{
     \scriptstyle k+l \text{~even}}
}

\title{Distinct Partitions and Some q-Binomial Summation Identities}
\author{M.J. Kronenburg}
\date{}

\begin{document}

\maketitle

\begin{abstract}
The partition functions $P(n,m,p)$, the number of integer partitions of $n$ into
exactly $m$ parts with each part at most $p$, and $Q(n,m,p)$, the number of integer
partitons of $n$ into exactly $m$ distinct parts with each part at most $p$,
are related by double summation identities which follow from their generating functions.
From these identities and some identities from an earlier paper,
some other identities involving distinct partitions
and some q-binomial summation identities are proved,
and from these follow some combinatorial identities.
\end{abstract}


\section{Introduction}

The following q-binomial summation identities are proved,
of which (\ref{result1}) is known as identity (1.8) and (\ref{result5}) as (1.9) in \cite{guo} and both in \cite{merca1},
and (\ref{result2}) is known as corollary 4.1 and (\ref{result6}) as corollary 3.7 in \cite{merca1}.
Corollaries 2.4 and 3.4 of \cite{merca1} follow from (\ref{result8}) and (\ref{result9}).
\begin{equation}
 \sum_{k=0}^n (-1)^k q^{\textstyle\binom{k}{2}} \gaussian{m+n-k}{m}\gaussian{m+1}{k} = \delta_{n,0}
\end{equation}
\begin{equation}\label{result1}
 \sum_{k=0}^{\lfloor n/2\rfloor} q^{\textstyle\binom{n-2k}{2}} \gaussian{m+1}{n-2k} \gaussiansqr{m+k}{m} = \gaussian{m+n}{m}
\end{equation}
\begin{equation}\label{result2}
 \sum_{k=0}^{\lfloor n/2\rfloor} (-1)^k q^{2\textstyle\binom{k}{2}} \gaussian{m+n-2k}{m} \gaussiansqr{m+1}{k} = q^{\textstyle\binom{n}{2}}\gaussian{m+1}{n}
\end{equation}
\begin{equation}\label{result3}
 \sum_{k=0}^{\lfloor n/3\rfloor} (-1)^k q^{\textstyle\binom{n-3k}{2}} \gaussian{m+1}{n-3k} \gaussiancube{m+k}{m} 
  = \sum_{k=0}^n \cos(\frac{2k-n}{3}\pi) \gaussian{m+n-k}{m} \gaussian{m+k}{m} \\
\end{equation}
\begin{equation}\label{result4}
  \sum_{k=0}^{\lfloor n/3\rfloor} (-1)^k q^{3\textstyle\binom{k}{2}} \gaussian{m+n-3k}{m} \gaussiancube{m+1}{k} 
  =  \sum_{k=0}^n \cos(\frac{2k-n}{3}\pi) q^{{\textstyle\binom{n-k}{2}}+\textstyle\binom{k}{2}} \gaussian{m+1}{n-k} \gaussian{m+1}{k} 
\end{equation}
\begin{equation}\label{result5}
 \sum_{k=0}^{\lfloor n/4\rfloor} q^{\textstyle\binom{n-4k}{2}} \gaussian{m+1}{n-4k} \gaussianquart{m+k}{m}
 = \sum_{k=0}^{\lfloor n/2\rfloor} (-1)^k \gaussian{m+n-2k}{m} \gaussiansqr{m+k}{m}
\end{equation}
\begin{equation}\label{result6}
 \sum_{k=0}^{\lfloor n/4\rfloor} (-1)^k q^{4\textstyle\binom{k}{2}} \gaussian{m+n-4k}{m} \gaussianquart{m+1}{k}
 = \sum_{k=0}^{\lfloor n/2\rfloor} q^{{\textstyle\binom{n-2k}{2}}+2\textstyle\binom{k}{2}} \gaussian{m+1}{n-2k} \gaussiansqr{m+1}{k}
\end{equation}
\begin{equation}\label{result7}
 \sum_{k=0}^n\sum_{l=0}^{n-k} (-1)^k 
  q^{a{\textstyle\binom{n-k-l}{2}}+b{\textstyle\binom{k}{2}}} \gaussianpow{p+n-k-l}{p}{c}\gaussianpow{m+1}{k}{b}\gaussianpow{m+l}{m}{b} 
  = q^{a\textstyle\binom{n}{2}}\gaussianpow{p+n}{p}{c}
\end{equation}
\begin{equation}\label{result8}
 \sum_{k=0}^n\sum_{l=0}^{n-k} (-1)^l
  q^{a{\textstyle\binom{n-k-l}{2}}+b{\textstyle\binom{k}{2}}} \gaussianpow{p+n-k-l}{p}{c}\gaussianpow{m+1}{k}{b}\gaussianpow{m+l}{m}{b} 
  = q^{a\textstyle\binom{n}{2}}\gaussianpow{p+n}{p}{c}
\end{equation}
\begin{equation}\label{result9}
 \sum_{k=0}^n\sum_{l=0}^{n-k} (-1)^k 
  q^{a{\textstyle\binom{n-k-l}{2}}+b{\textstyle\binom{k}{2}}} \gaussianpow{p}{n-k-l}{c}\gaussianpow{m+1}{k}{b}\gaussianpow{m+l}{m}{b} 
  = q^{a\textstyle\binom{n}{2}}\gaussianpow{p}{n}{c}
\end{equation}
\begin{equation}\label{result10}
 \sum_{k=0}^n\sum_{l=0}^{n-k} (-1)^l 
  q^{a{\textstyle\binom{n-k-l}{2}}+b{\textstyle\binom{k}{2}}} \gaussianpow{p}{n-k-l}{c}\gaussianpow{m+1}{k}{b}\gaussianpow{m+l}{m}{b} 
  = q^{a\textstyle\binom{n}{2}}\gaussianpow{p}{n}{c}
\end{equation}
In the summands of the last four identities, because of the type of double summation,
$k$ and $l$ can be interchanged, and $l$ can be replaced by $n-k-l$.
When $q=1$ these identities give the following combinatorial identities,
of which (\ref{binomres2}) and (\ref{binomres3}) are known as (3.24) and (3.25) in \cite{G72}.
\begin{equation}
 \sum_{k=0}^n (-1)^k \binom{m+k}{m} \binom{m+1}{n-k} = \delta_{n,0}
\end{equation}
\begin{equation}\label{binomres2}
 \sum_{k=0}^n \binom{m+1}{2k} \binom{m+n-k}{m} = \binom{m+2n}{m}
\end{equation}
\begin{equation}\label{binomres3}
 \sum_{k=0}^n \binom{m+1}{2k+1} \binom{m+n-k}{m} = \binom{m+2n+1}{m}
\end{equation}
\begin{equation}
 \sum_{k=0}^n (-1)^k \binom{m+2k}{m} \binom{m+1}{n-k} = (-1)^n \binom{m+1}{2n}
\end{equation}
\begin{equation}
 \sum_{k=0}^n (-1)^k \binom{m+2k+1}{m} \binom{m+1}{n-k} = (-1)^n \binom{m+1}{2n+1}
\end{equation}
\begin{equation}
 \sum_{k=0}^n (-1)^k \binom{m+1}{3k} \binom{m+n-k}{m} = \sum_{k=0}^{3n} \cos(\frac{2k}{3}\pi)\binom{m+3n-k}{m}\binom{m+k}{m}
\end{equation}
\begin{equation}
 \sum_{k=0}^n (-1)^k \binom{m+1}{3k+1} \binom{m+n-k}{m} = \sum_{k=0}^{3n+1} \cos(\frac{2k-1}{3}\pi)\binom{m+3n-k+1}{m}\binom{m+k}{m}
\end{equation}
\begin{equation}
 \sum_{k=0}^n (-1)^k \binom{m+1}{3k+2} \binom{m+n-k}{m} = \sum_{k=0}^{3n+2} \cos(\frac{2k-2}{3}\pi)\binom{m+3n-k+2}{m}\binom{m+k}{m}
\end{equation}
\begin{equation}
 \sum_{k=0}^n (-1)^k \binom{m+3k}{m} \binom{m+1}{n-k} = \sum_{k=0}^{3n} \cos(\frac{2k}{3}\pi) \binom{m+1}{3n-k}\binom{m+1}{k}
\end{equation}
\begin{equation}
 \sum_{k=0}^n (-1)^k \binom{m+3k+1}{m} \binom{m+1}{n-k} = \sum_{k=0}^{3n+1} \cos(\frac{2k-1}{3}\pi) \binom{m+1}{3n-k+1}\binom{m+1}{k}
\end{equation}
\begin{equation}
 \sum_{k=0}^n (-1)^k \binom{m+3k+2}{m} \binom{m+1}{n-k} = \sum_{k=0}^{3n+2} \cos(\frac{2k-2}{3}\pi) \binom{m+1}{3n-k+2}\binom{m+1}{k}
\end{equation}
\begin{equation}
 \sum_{k=0}^n \binom{m+1}{4k} \binom{m+n-k}{m} = \sum_{k=0}^{2n} (-1)^k \binom{m+2k}{m} \binom{m+2n-k}{m}
\end{equation}
\begin{equation}
 \sum_{k=0}^n \binom{m+1}{4k+1} \binom{m+n-k}{m} = \sum_{k=0}^{2n} (-1)^k \binom{m+2k+1}{m} \binom{m+2n-k}{m}
\end{equation}
\begin{equation}
 \sum_{k=0}^n \binom{m+1}{4k+2} \binom{m+n-k}{m} = \sum_{k=0}^{2n+1} (-1)^{k+1} \binom{m+2k}{m} \binom{m+2n-k+1}{m}
\end{equation}
\begin{equation}
 \sum_{k=0}^n \binom{m+1}{4k+3} \binom{m+n-k}{m} = \sum_{k=0}^{2n+1} (-1)^{k+1} \binom{m+2k+1}{m} \binom{m+2n-k+1}{m}
\end{equation}
\begin{equation}
 \sum_{k=0}^n (-1)^k \binom{m+4k}{m} \binom{m+1}{n-k} = (-1)^n \sum_{k=0}^{2n} \binom{m+1}{2k} \binom{m+1}{2n-k}
\end{equation}
\begin{equation}
 \sum_{k=0}^n (-1)^k \binom{m+4k+1}{m} \binom{m+1}{n-k} = (-1)^n \sum_{k=0}^{2n} \binom{m+1}{2k+1} \binom{m+1}{2n-k}
\end{equation}
\begin{equation}
 \sum_{k=0}^n (-1)^k \binom{m+4k+2}{m} \binom{m+1}{n-k} = (-1)^n \sum_{k=0}^{2n+1} \binom{m+1}{2k} \binom{m+1}{2n-k+1}
\end{equation}
\begin{equation}
 \sum_{k=0}^n (-1)^k \binom{m+4k+3}{m} \binom{m+1}{n-k} = (-1)^n \sum_{k=0}^{2n+1} \binom{m+1}{2k+1} \binom{m+1}{2n-k+1}
\end{equation}
\begin{equation}
 \sum_{k=0}^n\sum_{l=0}^{n-k} (-1)^k
  \binom{p+n-k-l}{p}\binom{m+1}{k}\binom{m+l}{m} = \binom{p+n}{p}
\end{equation}
\begin{equation}
 \sum_{k=0}^n\sum_{l=0}^{n-k} (-1)^l
  \binom{p+n-k-l}{p}\binom{m+1}{k}\binom{m+l}{m} = \binom{p+n}{p}
\end{equation}
\begin{equation}
 \sum_{k=0}^n\sum_{l=0}^{n-k} (-1)^k
  \binom{p}{n-k-l}\binom{m+1}{k}\binom{m+l}{m} = \binom{p}{n}
\end{equation}
\begin{equation}
 \sum_{k=0}^n\sum_{l=0}^{n-k} (-1)^l
  \binom{p}{n-k-l}\binom{m+1}{k}\binom{m+l}{m} = \binom{p}{n}
\end{equation}

\section{Definitions and Basic Identities}

Let the coefficient of a power series be defined as:
\begin{equation}
 [q^n] \sum_{k=0}^{\infty} a_k q^k = a_n
\end{equation}
Let $P(n)$ be the number of integer partitions of $n$,
let $Q(n)$ be the number of integer partitions of $n$ into distinct parts, let $P(n,m)$ be the
number of integer partitions of $n$ into exactly $m$ parts,
and let $Q(n,m)$ be the number of integer partitions of $n$ into exactly $m$ distinct parts.
Let $P(n,m,p)$ be the number of integer partitions of $n$ into exactly $m$ parts, each part at most $p$,
and let $P^*(n,m,p)$ be the number of integer partitions of $n$ into at most $m$ parts, each part at most $p$,
which is the number of Ferrer diagrams that fit in a $m$ by $p$ rectangle:
\begin{equation}\label{pnmpsum}
 P^*(n,m,p) = \sum_{k=0}^m P(n,k,p)
\end{equation}
Let the following definition of the q-binomial coefficient,
also called the Gaussian polynomial, be given.
\begin{definition}
The q-binomial coefficient is defined by \cite{A84,AAR}:
\begin{equation}\label{gaussdef}
 \gaussian{m+p}{m} = \prod_{j=1}^m \frac{1-q^{p+j}}{1-q^j}
\end{equation}
\end{definition}
The q-binomial coefficient is the generating function of $P^*(n,m,p)$ \cite{A84}:
\begin{equation}\label{pstar}
 P^*(n,m,p) = [q^n] \gaussian{m+p}{m}
\end{equation}
In the earlier paper \cite{MK2} it was proved that:
\begin{equation}\label{pnmpeq}
 P^*(n,m,p) = P(n+m,m,p+1)
\end{equation}
For the q-binomial coefficient there is the following symmetry identity \cite{MK2}:
\begin{equation}\label{symm}
 \gaussian{m+p}{m} = \gaussian{m+p}{p}
\end{equation}

\section{Formulas Involving Distinct Partitions}

Let $Q(n,m,p)$ be the number of integer partitions of $n$ into exactly $m$
distinct parts with each part at most $p$.
In the earlier paper \cite{MK2} it was proved that:
\begin{equation}\label{qnmpeq}
 Q(n,m,p) = P(n-m(m-1)/2,m,p-m+1)
\end{equation}
The generating functions for $P(n,m,p)$ and $Q(n,m,p)$ are identities (7.3) and (7.4) in \cite{AE04}:
\begin{equation}\label{gen1}
 \prod_{j=1}^p \frac{1}{1-zq^j} = \sum_{n=0}^{\infty}\sum_{m=0}^{\infty} P(n,m,p) q^n z^m
\end{equation}
\begin{equation}\label{gen2}
 \prod_{j=1}^p (1+zq^j) = \sum_{n=0}^{\infty}\sum_{m=0}^{\infty} Q(n,m,p) q^n z^m
\end{equation}
\begin{theorem}\label{theorem1}
\begin{equation}
 P(n,m,p) = \sum_{k=0}^{\lfloor n/2\rfloor}\sum_{l=0}^{\lfloor m/2\rfloor}
  Q(n-2k,m-2l,p) P(k,l,p)
\end{equation}
\end{theorem}
\begin{proof}
Using $(1+zq^j)(1-zq^j)=1-z^2q^{2j}$:
\begin{equation}
 \frac{1}{\prod_{j=1}^p(1-zq^j)} = \frac{\prod_{j=1}^p(1+zq^j)}{\prod_{j=1}^p(1-z^2q^{2j})}
\end{equation}
Substituting the generating functions (\ref{gen1}) and (\ref{gen2}):
\begin{equation}
\begin{split}
 & \sum_{n=0}^{\infty}\sum_{m=0}^{\infty}P(n,m,p)q^nz^m
  = ( \sum_{n=0}^{\infty}\sum_{m=0}^{\infty}Q(n,m,p)q^nz^m )(\sum_{n=0}^{\infty}\sum_{m=0}^{\infty}P(n,m,p)q^{2n}z^{2m} ) \\
 & = \sum_{n_1=0}^{\infty}\sum_{n_2=0}^{\infty}\sum_{m_1=0}^{\infty}\sum_{m_2=0}^{\infty} 
  Q(n_1,m_1,p)P(n_2,m_2,p) q^{n_1+2n_2} z^{m_1+2m_2} \\ 
\end{split}
\end{equation}
The coefficients on both sides must be equal,
so $n_1+2n_2=n$ and $m_1+2m_2=m$, which is equivalent to $n_1=n-2n_2$ and $m_1=m-2m_2$:
\begin{equation}
\begin{split}
 P(n,m,p) & = \dblsum{n_1}{n_2}{n_1+2n_2=n}\dblsum{m_1}{m_2}{m_1+2m_2=m} Q(n_1,m_1,p)P(n_2,m_2,p) \\
 & = \sum_{n_2=0}^{\lfloor n/2\rfloor}\sum_{m_2=0}^{\lfloor m/2\rfloor} Q(n-2n_2,m-2m_2,p) P(n_2,m_2,p) \\
\end{split}
\end{equation}
\end{proof}
Let $Q^*(n,m,p)$ be the number of integer partitions of $n$ into at most $m$ distinct parts
with each part at most $p$, which is defined like $P^*(n,m,p)$ in (\ref{pnmpsum}):
\begin{equation}
 Q^*(n,m,p) = \sum_{k=0}^m Q(n,k,p)
\end{equation}
From the previous theorem a relation between $Q^*(n,m,p)$ and $P(n,m,p)$ can be derived.
\begin{theorem}\label{theorem2}
\begin{equation}
 P(n+m,m,p+1) = \sum_{k=0}^{\lfloor n/2\rfloor}\sum_{l=0}^{\lfloor m/2\rfloor} Q^*(n-2k,m-2l,p)P(k,l,p)
\end{equation}
\end{theorem}
\begin{proof}
Using the previous theorem and (\ref{pnmpeq}):
\begin{equation}
\begin{split}
 P^*(n,m,p) & = P(n+m,m,p+1) = \sum_{h=0}^m P(n,h,p) \\
 & = \sum_{h=0}^m \sum_{k=0}^{\lfloor n/2\rfloor} \sum_{l=0}^{\lfloor h/2\rfloor} Q(n-2k,h-2l,p)P(k,l,p) \\
 & = \sum_{k=0}^{\lfloor n/2\rfloor} \sum_{l=0}^{\lfloor m/2\rfloor} \sum_{h=2l}^m  Q(n-2k,h-2l,p)P(k,l,p) \\
 & = \sum_{k=0}^{\lfloor n/2\rfloor} \sum_{l=0}^{\lfloor m/2\rfloor} Q^*(n-2k,m-2l,p)P(k,l,p) \\
\end{split}
\end{equation}
\end{proof}
Let $P_{\rm most}(n,p)$ be the number of integer partitions of $n$ with each part at most $p$.
From (\ref{pnmpsum}), (\ref{pnmpeq}) and conjugation of Ferrer diagrams \cite{MK2}:
\begin{equation}
 P_{\rm most}(n,p) = \sum_{k=0}^n P(n,k,p) = P^*(n,n,p) = P(2n,n,p+1) = P(n+p,p)
\end{equation}
Obviously $P^*(n,m,p)=P_{\rm most}(n,p)$ when $m\geq n$.
Let $Q_{\rm most}(n,p)$ be the number of integer partitions of $n$ into distinct parts with each part at most $p$,
for which $Q^*(n,m,p)=Q_{\rm most}(n,p)$ when $m\geq n$.
When taking $m=n$ in this theorem for nonzero summands $l\leq k$ and therefore $n-2l\geq n-2k$:
\begin{equation}
 P(n+p,p) = \sum_{k=0}^{\lfloor n/2\rfloor} Q_{\rm most}(n-2k,p) P(p+k,p)
\end{equation}
From this identity follows as a special case when taking $p=n$,
and using $P(2n,n)=P(n)$ and from \cite{MK} $P(n,m)=P(n-m)$ if $2m\geq n$
and therefore $P(n+k,n)=P(k)$ if $n\geq k$:
\begin{equation}
 P(n) = \sum_{k=0}^{\lfloor n/2\rfloor} Q(n-2k)P(k)
\end{equation}
\begin{theorem}\label{theorem3}
\begin{equation}
 Q(n,m,p) = \sum_{k=0}^{\lfloor n/2\rfloor}\sum_{l=0}^{\lfloor m/2\rfloor}
  (-1)^l P(n-2k,m-2l,p) Q(k,l,p)
\end{equation}
\end{theorem}
\begin{proof}
Using $(1+izq^j)(1-izq^j)=1+z^2q^{2j}$:
\begin{equation}
 \prod_{j=1}^p(1+izq^j) = \frac{\prod_{j=1}^p(1+z^2q^{2j})}{\prod_{j=1}^p(1-izq^j)}
\end{equation}
Substituting the generating functions (\ref{gen1}) and (\ref{gen2}):
\begin{equation}
\begin{split}
 & \sum_{n=0}^{\infty}\sum_{m=0}^{\infty}Q(n,m,p)i^mq^nz^m
  = ( \sum_{n=0}^{\infty}\sum_{m=0}^{\infty}P(n,m,p)i^mq^nz^m )(\sum_{n=0}^{\infty}\sum_{m=0}^{\infty}Q(n,m,p)q^{2n}z^{2m} ) \\
 & = \sum_{n_1=0}^{\infty}\sum_{n_2=0}^{\infty}\sum_{m_1=0}^{\infty}\sum_{m_2=0}^{\infty} 
  P(n_1,m_1,p)Q(n_2,m_2,p) i^{m1}q^{n_1+2n_2} z^{m_1+2m_2} \\ 
\end{split}
\end{equation}
The coefficients on both sides must be equal,
so $n_1+2n_2=n$ and $m_1+2m_2=m$, which is equivalent to $n_1=n-2n_2$ and $m_1=m-2m_2$:
\begin{equation}
\begin{split}
 i^mQ(n,m,p) & = \dblsum{n_1}{n_2}{n_1+2n_2=n}\dblsum{m_1}{m_2}{m_1+2m_2=m} i^{m_1}P(n_1,m_1,p)Q(n_2,m_2,p) \\
 & = \sum_{n_2=0}^{\lfloor n/2\rfloor}\sum_{m_2=0}^{\lfloor m/2\rfloor} i^{m-2m_2}P(n-2n_2,m-2m_2,p) Q(n_2,m_2,p) \\
\end{split}
\end{equation}
With $i^{-2m_2}=(-1)^{m_2}$ the theorem is proved.
\end{proof}
Using a similar derivation as in theorem \ref{theorem2} gives:
\begin{equation}
 Q^*(n,m,p) = \sum_{k=0}^{\lfloor n/2\rfloor}\sum_{l=0}^{\lfloor m/2\rfloor} (-1)^l P(n+m-2(k+l),m-2l,p+1)Q(k,l,p)
\end{equation}
Using a similar reasoning as above:
\begin{equation}
 Q(n) = \sum_{k=0}^{\lfloor n/2\rfloor}\sum_{l=0}^{\lfloor n/2\rfloor} (-1)^l P(n-2k) Q(k,l)
\end{equation}
\begin{theorem}\label{theorem6}
\begin{equation}
 \sum_{k=0}^{\lfloor n/3\rfloor}\sum_{l=0}^{\lfloor m/3\rfloor} (-1)^l Q(n-3k,m-3l,p)P(k,l,p)
 = \sum_{k=0}^n\sum_{l=0}^m \cos(\frac{2l-m}{3}\pi) P(n-k,m-l,p) P(k,l,p)
\end{equation}
\end{theorem}
\begin{proof}
Using $(1-zq^j)(1-(-1)^{2/3}zq^j)(1+(-1)^{1/3}zq^j)=1-z^3q^{3j}$:
\begin{equation}
 \frac{\prod_{j=1}^p(1+(-1)^{1/3}zq^j)}{\prod_{j=1}^p(1-z^3q^{3j})} 
 = \frac{1}{\prod_{j=1}^p(1-zq^j)\prod_{j=1}^p(1-(-1)^{2/3}zq^j)} 
\end{equation}
Substituting the generating functions (\ref{gen1}) and (\ref{gen2}):
\begin{equation}
\begin{split}
  \sum_{n_1=0}^{\infty}\sum_{n_2=0}^{\infty}\sum_{m_1=0}^{\infty}\sum_{m_2=0}^{\infty} 
  Q(n_1,m_1,p)P(n_2,m_2,p) (-1)^{m_1/3}q^{n_1+3n_2} z^{m_1+3m_2} \\ 
 = \sum_{n_3=0}^{\infty}\sum_{n_4=0}^{\infty}\sum_{m_3=0}^{\infty}\sum_{m_4=0}^{\infty} 
  P(n_3,m_3,p)P(n_4,m_4,p) (-1)^{2m_3/3}q^{n_3+n_4} z^{m_3+m_4} \\ 
\end{split}
\end{equation}
Taking the coefficients on both sides equal to $q^nz^m$, then $n_1+3n_2=n_3+n_4=n$ and $m_1+3m_2=m_3+m_4=m$,
which is equivalent to $n_1=n-3n_2$, $n_3=n-n_4$, $m_1=m-3m_2$ and $m_3=m-m_4$, gives:
\begin{equation}
 \sum_{k=0}^{\lfloor n/3\rfloor}\sum_{l=0}^{\lfloor m/3\rfloor} (-1)^l Q(n-3k,m-3l,p)P(k,l,p)
 = \sum_{k=0}^n\sum_{l=0}^m (-1)^{\frac{m-2l}{3}} P(n-k,m-l,p) P(k,l,p)
\end{equation}
Equating the imaginary parts of this identity gives:
\begin{equation}
 \sum_{k=0}^n\sum_{l=0}^m \sin(\frac{m-2l}{3}\pi) P(n-k,m-l,p) P(k,l,p) = 0
\end{equation}
Changing $k$ into $n-k$ and $l$ into $m-l$ changes the sign of the summand,
and therefore this identity is trivial.
Equating the real parts of the previous identity and using $\cos(x)=\cos(-x)$ gives the theorem.
\end{proof}
\begin{theorem}\label{theorem7}
\begin{equation}
 \sum_{k=0}^{\lfloor n/3\rfloor}\sum_{l=0}^{\lfloor m/3\rfloor} (-1)^l P(n-3k,m-3l,p)Q(k,l,p)
 = \sum_{k=0}^n\sum_{l=0}^m \cos(\frac{2l-m}{3}\pi) Q(n-k,m-l,p) Q(k,l,p) 
\end{equation}
\end{theorem}
\begin{proof}
In the previous theorem replacing $z$ by $-z$ leads to:
\begin{equation}
 \frac{\prod_{j=1}^p(1+z^3q^{3j})}{\prod_{j=1}^p(1-(-1)^{1/3}zq^j)} = \prod_{j=1}^p(1+zq^j)\prod_{j=1}^p(1+(-1)^{2/3}zq^j)
\end{equation}
When comparing this with the previous theorem it is clear that $P$ and $Q$ are interchanged,
which gives:
\begin{equation}
 \sum_{k=0}^{\lfloor n/3\rfloor}\sum_{l=0}^{\lfloor m/3\rfloor} (-1)^l P(n-3k,m-3l,p)Q(k,l,p)
 = \sum_{k=0}^n\sum_{l=0}^m (-1)^{\frac{m-2l}{3}} Q(n-k,m-l,p) Q(k,l,p)
\end{equation}
Equating the imaginary parts of this identity gives:
\begin{equation}
 \sum_{k=0}^n\sum_{l=0}^m \sin(\frac{m-2l}{3}\pi) Q(n-k,m-l,p) Q(k,l,p) = 0
\end{equation}
As in the previous theorem this identity is trivial.
Equating the real parts of the previous identity and using $\cos(x)=\cos(-x)$ gives the theorem.
\end{proof}
\begin{theorem}\label{theorem8}
\begin{equation}
 \sum_{k=0}^{\lfloor n/4\rfloor}\sum_{l=0}^{\lfloor m/4\rfloor} Q(n-4k,m-4l,p)P(k,l,p)
 = \sum_{k=0}^{\lfloor n/2\rfloor}\sum_{l=0}^{\lfloor m/2\rfloor} (-1)^l P(n-2k,m-2l,p) P(k,l,p)  
\end{equation}
\end{theorem}
\begin{proof}
Using $(1+izq^j)(1-izq^j)(1-z^2q^{2j})=1-z^4q^{4j}$:
\begin{equation}
 \frac{\prod_{j=1}^p(1+izq^j)}{\prod_{j=1}^p(1-z^4q^{4j})} 
 = \frac{1}{\prod_{j=1}^p(1-izq^j)\prod_{j=1}^p(1-z^2q^{2j})} 
\end{equation}
The proof is similar to the previous proofs.
\end{proof}
\begin{theorem}\label{theorem9}
\begin{equation}
 \sum_{k=0}^{\lfloor n/4\rfloor}\sum_{l=0}^{\lfloor m/4\rfloor} (-1)^l P(n-4k,m-4l,p)Q(k,l,p)
 = \sum_{k=0}^{\lfloor n/2\rfloor}\sum_{l=0}^{\lfloor m/2\rfloor} Q(n-2k,m-2l,p) Q(k,l,p)  
\end{equation}
\end{theorem}
\begin{proof}
Using $(1+(-1)^{1/4}zq^j)(1-(-1)^{1/4}zq^j)(1+iz^2q^{2j})=1+z^4q^{4j}$:
\begin{equation}
 \frac{\prod_{j=1}^p(1+z^4q^{4j})}{\prod_{j=1}^p(1-(-1)^{1/4}zq^j)} 
 = \prod_{j=1}^p(1+(-1)^{1/4}zq^j)\prod_{j=1}^p(1+iz^2q^{2j})
\end{equation}
The proof is similar to the previous proofs.
\end{proof}
\begin{theorem}\label{simple}
\begin{equation}
 \sum_{k=0}^n\sum_{l=0}^m (-1)^l P(n-k,m-l,p)Q(k,l,p) = \delta_{n,0}\delta_{m,0}
\end{equation}
\end{theorem}
\begin{proof}
\begin{equation}
 \frac{\prod_{j=1}^p(1+(-z)q^j)}{\prod_{j=1}^p(1-zq^j)} = 1
\end{equation}
\begin{equation}
\begin{split}
 & \sum_{n=0}^{\infty}\sum_{m=0}^{\infty} \delta_{n,0}\delta_{m,0} q^nz^m
  = (\sum_{n=0}^{\infty}\sum_{m=0}^{\infty}P(n,m,p)q^nz^m)(\sum_{n=0}^{\infty}\sum_{m=0}^{\infty}Q(n,m,p)(-1)^mq^nz^m) \\
 & = \sum_{n_1=0}^{\infty}\sum_{n_2=0}^{\infty}\sum_{m_1=0}^{\infty}\sum_{m_2=0}^{\infty}
   P(n_1,m_1,p)Q(n_2,m_2,p) (-1)^{m_2} q^{n_1+n_2} z^{m_1+m_2} \\
\end{split}
\end{equation}
The coefficients on both sides must be equal, so $n_1+n_2=n$ and $m_1+m_2=m$,
which is equivalent to $n_1=n-n_2$ and $m_1=m-m_2$, which gives the theorem. 
\end{proof}

\section{Some q-Binomial Summation Identities}

From theorems \ref{theorem1}, \ref{theorem3}, \ref{theorem6}, \ref{theorem7},
\ref{theorem8}, \ref{theorem9} and \ref{simple}
the following q-binomial summation identities are proved.
\begin{theorem}\label{theorem4}
\begin{equation}
 \sum_{k=0}^{\lfloor n/2\rfloor} q^{\textstyle\binom{n-2k}{2}} \gaussian{m+1}{n-2k} \gaussiansqr{m+k}{m} = \gaussian{m+n}{m}
\end{equation}
\end{theorem}
\begin{proof}
From theorem \ref{theorem1} using (\ref{pnmpeq}) and (\ref{qnmpeq}):
\begin{equation}
 P^*(n-m,m,p-1) = \sum_{k=0}^{\lfloor n/2\rfloor}\sum_{l=0}^{\lfloor m/2\rfloor}
  P^*(n-2k-(m-2l)(m-2l+1)/2,m-2l,p-m+2l)P^*(k-l,l,p-1)
\end{equation}
Using (\ref{pstar}):
\begin{equation}
\begin{split}
 [q^{n-m}]\gaussian{m+p-1}{m} & = \sum_{k=0}^{\lfloor n/2\rfloor}\sum_{l=0}^{\lfloor m/2\rfloor} 
 [q^{n-(m-2l)(m-2l+1)/2-2k}]\gaussian{p}{m-2l} \cdot [q^{k-l}]\gaussian{p+l-1}{l} \\
   & = \sum_{k=0}^{\lfloor n/2\rfloor}\sum_{l=0}^{\lfloor m/2\rfloor} 
 [q^{n-2k}]q^{(m-2l)(m-2l+1)/2}\gaussian{p}{m-2l} \cdot [q^k]q^l\gaussian{p+l-1}{l} \\
\end{split}
\end{equation}
In theorem \ref{theorem1} it was used that:
\begin{equation}
 \sum_{k=0}^{\lfloor n/2\rfloor} a_{n-2k}b_k = [q^n] (\sum_{k=0}^{\infty} a_kq^k)(\sum_{k=0}^{\infty}b_kq^{2k})
\end{equation}
so the summation over $k$ can be done:
\begin{equation}
 [q^n]q^m\gaussian{m+p-1}{m} = [q^n] \sum_{l=0}^{\lfloor m/2\rfloor} 
  q^{(m-2l)(m-2l+1)/2}\gaussian{p}{m-2l} q^{2l}\gaussiansqr{p+l-1}{l}
\end{equation}
Because all coefficients $[q^n]$ are equal, the polynomials must be equal,
and cancelling some powers of $q$:
\begin{equation}
 \gaussian{m+p-1}{m} = \sum_{l=0}^{\lfloor m/2\rfloor} 
   q^{(m-2l)(m-2l-1)/2}\gaussian{p}{m-2l} \gaussiansqr{p+l-1}{l} 
\end{equation}
Replacing $m$ by $n$ and $l$ by $k$ and $p$ by $m+1$
and using (\ref{symm}) gives the theorem.
\end{proof}
Taking $q=1$ and replacing $n$ by $2n$ and $k$ by $n-k$ this is combinatorial identity (3.24) in \cite{G72}:
\begin{equation}
 \sum_{k=0}^n \binom{m+1}{2k} \binom{m+n-k}{m} = \binom{m+2n}{m}
\end{equation}
and replacing $n$ by $2n+1$ and $k$ by $n-k$ this is combinatorial identity (3.25) in \cite{G72}
\begin{equation}
 \sum_{k=0}^n \binom{m+1}{2k+1} \binom{m+n-k}{m} = \binom{m+2n+1}{m}
\end{equation}
\begin{theorem}
\begin{equation}
 \sum_{k=0}^{\lfloor n/2\rfloor} (-1)^k q^{2\textstyle\binom{k}{2}} \gaussian{m+n-2k}{m} \gaussiansqr{m+1}{k} 
  = q^{\textstyle\binom{n}{2}}\gaussian{m+1}{n}
\end{equation}
\end{theorem}
\begin{proof}
From theorem \ref{theorem3} using (\ref{pnmpeq}) and (\ref{qnmpeq}):
\begin{equation}
\begin{split}
 & P^*(n-m(m+1)/2,m,p-m) \\
 & = \sum_{k=0}^{\lfloor n/2\rfloor}\sum_{l=0}^{\lfloor m/2\rfloor}
  (-1)^l P^*(n-2k-m+2l,m-2l,p-1)P^*(k-l(l+1)/2,l,p-l) \\
\end{split}
\end{equation}
Using (\ref{pstar}) and (\ref{symm}):
\begin{equation}
\begin{split}
 [q^{n-m(m+1)/2}]\gaussian{p}{m} & = \sum_{k=0}^{\lfloor n/2\rfloor}\sum_{l=0}^{\lfloor m/2\rfloor} 
 (-1)^l [q^{n-m-2k+2l}]\gaussian{m+p-2l-1}{m-2l} \cdot [q^{k-l(l+1)/2}]\gaussian{p}{l} \\
   & = \sum_{k=0}^{\lfloor n/2\rfloor}\sum_{l=0}^{\lfloor m/2\rfloor} 
 (-1)^l [q^{n-2k}]q^{m-2l}\gaussian{m+p-2l-1}{p-1} \cdot [q^k]q^{l(l+1)/2}\gaussian{p}{l} \\
\end{split}
\end{equation}
As in theorem \ref{theorem4} the sum over $k$ can be done:
\begin{equation}
 [q^n]q^{m(m+1)/2}\gaussian{p}{m} = [q^n] \sum_{l=0}^{\lfloor m/2\rfloor} 
  (-1)^l q^{m-2l}\gaussian{m+p-2l-1}{p-1} q^{l(l+1)}\gaussiansqr{p}{l}
\end{equation}
Because all coefficients $[q^n]$ are equal, the polynomials must be equal,
and cancelling some powers of $q$:
\begin{equation}
 q^{m(m-1)/2}\gaussian{p}{m} = \sum_{l=0}^{\lfloor m/2\rfloor} 
  (-1)^l q^{l(l-1)}\gaussian{m+p-2l-1}{p-1} \gaussiansqr{p}{l}
\end{equation}
Replacing $m$ by $n$ and $l$ by $k$ and $p$ by $m+1$ gives the theorem.
\end{proof}
Taking $q=1$ and replacing $n$ by $2n$ and $k$ by $n-k$ gives the combinatorial identity:
\begin{equation}
 \sum_{k=0}^n (-1)^k \binom{m+2k}{m} \binom{m+1}{n-k} = (-1)^n \binom{m+1}{2n}
\end{equation}
and replacing $n$ by $2n+1$ and $k$ by $n-k$ gives the combinatorial identity:
\begin{equation}
 \sum_{k=0}^n (-1)^k \binom{m+2k+1}{m} \binom{m+1}{n-k} = (-1)^n \binom{m+1}{2n+1}
\end{equation}
\begin{theorem}
\begin{equation}
  \sum_{k=0}^{\lfloor n/3\rfloor} (-1)^k q^{\textstyle\binom{n-3k}{2}} \gaussian{m+1}{n-3k} \gaussiancube{m+k}{m} 
   = \sum_{k=0}^n \cos(\frac{2k-n}{3}\pi) \gaussian{m+n-k}{m} \gaussian{m+k}{m} 
\end{equation}
\end{theorem}
\begin{proof}
From theorem \ref{theorem6} using (\ref{pnmpeq}) and (\ref{qnmpeq}):
\begin{equation}
\begin{split}
 & \sum_{k=0}^{\lfloor n/3\rfloor}\sum_{l=0}^{\lfloor m/3\rfloor} (-1)^l P^*(n-3k-(m-3l)(m-3l-1)/2,m-3l,p-m+3l) P^*(k-l,l,p-1) \\
 & = \sum_{k=0}^n \sum_{l=0}^m \cos(\frac{2l-m}{3}\pi) P^*(n-k-m+l,m-l,p-1) P^*(k-l,l,p-1) \\
\end{split}
\end{equation}
Using (\ref{pstar}) and (\ref{symm}):
\begin{equation}
\begin{split}
 & \sum_{k=0}^{\lfloor n/3\rfloor}\sum_{l=0}^{\lfloor m/3\rfloor} (-1)^l [q^{n-3k}]q^{(m-3l)(m-3l+1)/2} \gaussian{p}{m-3l}
 \cdot [q^k] q^l \gaussian{p+l-1}{p-1} \\
 & = \sum_{k=0}^n \sum_{l=0}^m \cos(\frac{2l-m}{3}\pi) [q^{n-k}] q^{m-l} \gaussian{p+m-l-1}{p-1} \cdot [q^k] q^l \gaussian{p+l-1}{p-1} \\
\end{split}
\end{equation}
In theorem \ref{theorem6} it was used that:
\begin{equation}
 \sum_{k=0}^{\lfloor n/3\rfloor} a_{n-3k}b_k = [q^n] (\sum_{k=0}^{\infty} a_kq^k)(\sum_{k=0}^{\infty}b_kq^{3k})
\end{equation}
so as in the previous theorems the summation over $k$ can be done:
\begin{equation}
\begin{split}
 & \sum_{l=0}^{\lfloor m/3\rfloor} (-1)^l q^{(m-3l)(m-3l+1)/2} \gaussian{p}{m-3l} q^{3l} \gaussiancube{p+l-1}{p-1} \\
 & = \sum_{l=0}^m \cos(\frac{2l-m}{3}\pi) q^{m-l} \gaussian{p+m-l-1}{p-1} q^l \gaussian{p+l-1}{p-1} \\
\end{split} 
\end{equation}
Cancelling some powers of $q$ and replacing $m$ by $n$ and $l$ by $k$ and $p$ by $m+1$ gives the theorem.
\end{proof}
Taking $q=1$ and replacing $n$ by $3n$ and in the left side
replacing $k$ by $n-k$ and using $\cos(\alpha-n\pi)=(-1)^n\cos(\alpha)$ gives the combinatorial identity:
\begin{equation}
 \sum_{k=0}^n (-1)^k \binom{m+1}{3k} \binom{m+n-k}{m} = \sum_{k=0}^{3n} \cos(\frac{2k}{3}\pi)\binom{m+3n-k}{m}\binom{m+k}{m}
\end{equation}
Replacing $n$ by $3n+1$ or $3n+2$ gives similar identities.
\begin{theorem}
\begin{equation}
  \sum_{k=0}^{\lfloor n/3\rfloor} (-1)^k q^{3\textstyle\binom{k}{2}} \gaussian{m+n-3k}{m} \gaussiancube{m+1}{k} 
  = \sum_{k=0}^n \cos(\frac{2k-n}{3}\pi) q^{{\textstyle\binom{n-k}{2}}+\textstyle\binom{k}{2}} \gaussian{m+1}{n-k} \gaussian{m+1}{k} 
\end{equation}
\end{theorem}
\begin{proof}
From theorem \ref{theorem7} using (\ref{pnmpeq}) and (\ref{qnmpeq}):
\begin{equation}
\begin{split}
 & \sum_{k=0}^{\lfloor n/3\rfloor}\sum_{l=0}^{\lfloor m/3\rfloor} (-1)^l P^*(n-3k-m+3l,m-3l,p-1) P^*(k-l(l+1)/2,l,p-l) \\
 & = \sum_{k=0}^n \sum_{l=0}^m \cos(\frac{2l-m}{3}\pi) P^*(n-k-(m-l)(m-l+1)/2,m-l,p-m+l) \\
 & \qquad\qquad\qquad \cdot P^*(k-l(l+1)/2,l,p-l) \\
\end{split}
\end{equation}
Using (\ref{pstar}) and (\ref{symm}):
\begin{equation}
\begin{split}
 & \sum_{k=0}^{\lfloor n/3\rfloor}\sum_{l=0}^{\lfloor m/3\rfloor} (-1)^l [q^{n-3k}]q^{m-3l} \gaussian{m+p-3l-1}{p-1}
 \cdot [q^k] q^{l(l+1)/2} \gaussian{p}{l} \\
 & = \sum_{k=0}^n \sum_{l=0}^m \cos(\frac{2l-m}{3}\pi) [q^{n-k}] q^{(m-l)(m-l+1)/2} \gaussian{p}{m-l} 
 \cdot [q^k] q^{l(l+1)/2} \gaussian{p}{l} \\
\end{split}
\end{equation}
As in the previous theorem the summation over $k$ can be done:
\begin{equation}
\begin{split}
 & \sum_{l=0}^{\lfloor m/3\rfloor} (-1)^l q^{m-3l} \gaussian{m+p-3l-1}{p-1} q^{3l(l+1)/2} \gaussiancube{p}{l} \\
 & = \sum_{l=0}^m \cos(\frac{2l-m}{3}\pi) q^{(m-l)(m-l+1)/2} \gaussian{p}{m-l} q^{l(l+1)/2} \gaussian{p}{l} \\
\end{split} 
\end{equation}
Cancelling some powers of $q$ 
and replacing $m$ by $n$ and $l$ by $k$ and $p$ by $m+1$ gives the theorem.
\end{proof}
Taking $q=1$ and replacing $n$ by $3n$ and in the left side
replacing $k$ by $n-k$ and using $\cos(\alpha-n\pi)=(-1)^n\cos(\alpha)$ gives the combinatorial identity:
\begin{equation}
 \sum_{k=0}^n (-1)^k \binom{m+3k}{m} \binom{m+1}{n-k} = \sum_{k=0}^{3n} \cos(\frac{2k}{3}\pi) \binom{m+1}{3n-k}\binom{m+1}{k}
\end{equation}
Replacing $n$ by $3n+1$ or $3n+2$ gives similar identities.
\begin{theorem}
\begin{equation}
 \sum_{k=0}^{\lfloor n/4\rfloor} q^{\textstyle\binom{n-4k}{2}} \gaussian{m+1}{n-4k} \gaussianquart{m+k}{m}
 = \sum_{k=0}^{\lfloor n/2\rfloor} (-1)^k \gaussian{m+n-2k}{m} \gaussiansqr{m+k}{m}
\end{equation}
\end{theorem}
\begin{proof}
From theorem \ref{theorem8} using (\ref{pnmpeq}) and (\ref{qnmpeq}):
\begin{equation}
\begin{split}
 & \sum_{k=0}^{\lfloor n/4\rfloor}\sum_{l=0}^{\lfloor m/4\rfloor} P^*(n-4k-(m-4l)(m-4l+1)/2,m-4l,p-m+4l) P^*(k-l,l,p-1) \\
 & = \sum_{k=0}^{\lfloor n/2\rfloor} \sum_{l=0}^{\lfloor m/2\rfloor} (-1)^l P^*(n-2k-m+2l,m-2l,p-1) P^*(k-l,l,p-1) \\
\end{split}
\end{equation}
Using (\ref{pstar}) and (\ref{symm}):
\begin{equation}
\begin{split}
 & \sum_{k=0}^{\lfloor n/4\rfloor}\sum_{l=0}^{\lfloor m/4\rfloor} [q^{n-4k}] q^{(m-4l)(m-4l+1)/2} \gaussian{p}{m-4l}
  \cdot [q^k] q^l \gaussian{p+l-1}{p-1} \\
 & = \sum_{k=0}^{\lfloor n/2\rfloor} \sum_{l=0}^{\lfloor m/2\rfloor} (-1)^l [q^{n-2k}] q^{m-2l} \gaussian{m+p-2l-1}{p-1}
  \cdot [q^k] q^l \gaussian{p+l-1}{p-1} \\
\end{split}
\end{equation}
As in the previous theorem the summation over $k$ can be done:
\begin{equation}
\begin{split}
 & \sum_{k=0}^{\lfloor n/4\rfloor}\sum_{l=0}^{\lfloor m/4\rfloor} q^{(m-4l)(m-4l+1)/2} \gaussian{p}{m-4l}
  q^{4l} \gaussianquart{p+l-1}{p-1} \\
 & = \sum_{k=0}^{\lfloor n/2\rfloor} \sum_{l=0}^{\lfloor m/2\rfloor} (-1)^l q^{m-2l} \gaussian{m+p-2l-1}{p-1}
  q^{2l} \gaussiansqr{p+l-1}{p-1} \\
\end{split}
\end{equation}
Cancelling some powers of $q$
and replacing $m$ by $n$ and $l$ by $k$ and $p$ by $m+1$ gives the theorem.
\end{proof}
Taking $q=1$ and replacing $n$ by $4n$ and in the left side
replacing $k$ by $n-k$ and in the right side $k$ by $2n-k$ gives the combinatorial identity:
\begin{equation}
 \sum_{k=0}^n \binom{m+1}{4k} \binom{m+n-k}{m} = \sum_{k=0}^{2n} (-1)^k \binom{m+2k}{m} \binom{m+2n-k}{m}
\end{equation}
Replacing $n$ by $4n+1$ or $4n+2$ or $4n+3$ gives similar identities.
\begin{theorem}
\begin{equation}
 \sum_{k=0}^{\lfloor n/4\rfloor} (-1)^k q^{4\textstyle\binom{k}{2}} \gaussian{m+n-4k}{m} \gaussianquart{m+1}{k}
 = \sum_{k=0}^{\lfloor n/2\rfloor} q^{{\textstyle\binom{n-2k}{2}}+2{\textstyle\binom{k}{2}}} \gaussian{m+1}{n-2k} \gaussiansqr{m+1}{k}
\end{equation}
\end{theorem}
\begin{proof}
From theorem \ref{theorem9} using (\ref{pnmpeq}) and (\ref{qnmpeq}):
\begin{equation}
\begin{split}
 & \sum_{k=0}^{\lfloor n/4\rfloor}\sum_{l=0}^{\lfloor m/4\rfloor} (-1)^l P^*(n-4k-m+4l,m-4l,p-1) P^*(k-l(l+1)/2,l,p-l) \\
 & = \sum_{k=0}^{\lfloor n/2\rfloor} \sum_{l=0}^{\lfloor m/2\rfloor} P^*(n-2k-(m-2l)(m-2l+1)/2,m-2l,p-m+2l) 
   P^*(k-l(l+1)/2,l,p-l) \\
\end{split}
\end{equation}
Using (\ref{pstar}) and (\ref{symm}):
\begin{equation}
\begin{split}
 & \sum_{k=0}^{\lfloor n/4\rfloor}\sum_{l=0}^{\lfloor m/4\rfloor} (-1)^l [q^{n-4k}] q^{m-4l} \gaussian{p+m-4l-1}{p-1}
  \cdot [q^k] q^{l(l+1)/2} \gaussian{p}{l} \\
 & = \sum_{k=0}^{\lfloor n/2\rfloor} \sum_{l=0}^{\lfloor m/2\rfloor} [q^{n-2k}] q^{(m-2l)(m-2l+1)/2} \gaussian{p}{m-2l}
  \cdot [q^k] q^{l(l+1)/2} \gaussian{p}{l} \\
\end{split}
\end{equation}
As in the previous theorem the summation over $k$ can be done:
\begin{equation}
\begin{split}
 & \sum_{k=0}^{\lfloor n/4\rfloor}\sum_{l=0}^{\lfloor m/4\rfloor} (-1)^l q^{m-4l} \gaussian{p+m-4l-1}{p-1}
  q^{2l(l+1)} \gaussianquart{p}{l} \\
 & = \sum_{k=0}^{\lfloor n/2\rfloor} \sum_{l=0}^{\lfloor m/2\rfloor} q^{(m-2l)(m-2l+1)/2} \gaussian{p}{m-2l}
  q^{l(l+1)} \gaussiansqr{p}{l} \\
\end{split}
\end{equation}
Cancelling some powers of $q$ 
and replacing $m$ by $n$ and $l$ by $k$ and $p$ by $m+1$ gives the theorem.
\end{proof}
Taking $q=1$ and replacing $n$ by $4n$ and in the left side
replacing $k$ by $n-k$ and in the right side $k$ by $2n-k$ gives the combinatorial identity:
\begin{equation}
 \sum_{k=0}^n (-1)^k \binom{m+4k}{m} \binom{m+1}{n-k} = (-1)^n \sum_{k=0}^{2n} \binom{m+1}{2k} \binom{m+1}{2n-k}
\end{equation}
Replacing $n$ by $4n+1$ or $4n+2$ or $4n+3$ gives similar identities.
\begin{theorem}\label{delta}
\begin{equation}
 \sum_{k=0}^n (-1)^k q^{\textstyle\binom{k}{2}} \gaussian{m+n-k}{m}\gaussian{m+1}{k} = \delta_{n,0}
\end{equation}
\end{theorem}
\begin{proof}
From theorem \ref{simple} using (\ref{pnmpeq}) and (\ref{qnmpeq}):
\begin{equation}
\begin{split}
 & \sum_{k=0}^n\sum_{l=0}^m (-1)^l P^*(n-k-m+l,m-l,p-1) P^*(k-l(l+1)/2,l,p-l) \\
 = & \sum_{k=0}^n\sum_{l=0}^m (-1)^l [q^{n-k}] q^{m-l} \gaussian{m+p-l-1}{p-1} \cdot [q^k] q^{l(l+1)/2} \gaussian{p}{l} \\
 = & [q^n] \sum_{l=0}^m (-1)^l q^{m+l(l-1)/2} \gaussian{m+p-l-1}{p-1} \gaussian{p}{l}  = \delta_{n,0}\delta_{m,0} \\
\end{split}
\end{equation}
Replacing $m$ by $n$ and $l$ by $k$ and $p$ by $m+1$ gives the theorem.
\end{proof}
Taking $q=1$ and replacing $k$ by $n-k$ gives the combinatorial identity:
\begin{equation}
 \sum_{k=0}^n (-1)^k \binom{m+k}{m} \binom{m+1}{n-k} = \delta_{n,0}
\end{equation}
\begin{theorem}
\begin{equation}
 \sum_{k=0}^n\sum_{l=0}^{n-k} (-1)^k F(k+l) q^{\textstyle\binom{k}{2}} \gaussian{m+1}{k}\gaussian{m+l}{m} = F(0)
\end{equation}
\begin{equation}
 \sum_{k=0}^n\sum_{l=0}^{n-k} (-1)^l F(k+l) q^{\textstyle\binom{k}{2}} \gaussian{m+1}{k}\gaussian{m+l}{m} = F(0)
\end{equation}
\end{theorem}
\begin{proof}
The double summation over $k$ and $l$ is over a triangle, and the sum over each diagonal $k+l=c$
is zero because of the previous theorem with $n=c$, except at the origin $k=l=c=0$,
where the summand is the right side of the identity.
\end{proof}
The following four identities are an application of this theorem.
\begin{equation}\label{resdbl1}
 \sum_{k=0}^n\sum_{l=0}^{n-k} (-1)^k 
  q^{a{\textstyle\binom{n-k-l}{2}}+b{\textstyle\binom{k}{2}}} \gaussianpow{p+n-k-l}{p}{c}\gaussianpow{m+1}{k}{b}\gaussianpow{m+l}{m}{b} 
  = q^{a\textstyle\binom{n}{2}}\gaussianpow{p+n}{p}{c}
\end{equation}
\begin{equation}\label{resdbl2}
 \sum_{k=0}^n\sum_{l=0}^{n-k} (-1)^l
  q^{a{\textstyle\binom{n-k-l}{2}}+b{\textstyle\binom{k}{2}}} \gaussianpow{p+n-k-l}{p}{c}\gaussianpow{m+1}{k}{b}\gaussianpow{m+l}{m}{b} 
  = q^{a\textstyle\binom{n}{2}}\gaussianpow{p+n}{p}{c}
\end{equation}
\begin{equation}\label{resdbl3}
 \sum_{k=0}^n\sum_{l=0}^{n-k} (-1)^k 
  q^{a{\textstyle\binom{n-k-l}{2}}+b{\textstyle\binom{k}{2}}} \gaussianpow{p}{n-k-l}{c}\gaussianpow{m+1}{k}{b}\gaussianpow{m+l}{m}{b} 
  = q^{a\textstyle\binom{n}{2}}\gaussianpow{p}{n}{c}
\end{equation}
\begin{equation}\label{resdbl4}
 \sum_{k=0}^n\sum_{l=0}^{n-k} (-1)^l 
  q^{a{\textstyle\binom{n-k-l}{2}}+b{\textstyle\binom{k}{2}}} \gaussianpow{p}{n-k-l}{c}\gaussianpow{m+1}{k}{b}\gaussianpow{m+l}{m}{b} 
  = q^{a\textstyle\binom{n}{2}}\gaussianpow{p}{n}{c}
\end{equation}
In the summands of the last four identities, because of the type of double summation,
$k$ and $l$ can be interchanged and $l$ can be replaced by $n-k-l$.
Some pairs of these identities or identities derived from them in this way
have identical summands if $(-1)^k$ is replaced by $(-1)^l$ and vice versa
and have identical right sides.
In these cases a linear combination of the two identities can be
taken such that in the summand of the new identity:
\begin{equation}\label{linsum1}
 \frac{1}{2} [(-1)^k+(-1)^l] =
\begin{cases}
 (-1)^k & $\text{if $k+l$ is even}$ \\
 0      & $\text{if $k+l$ is odd}$ \\
\end{cases}
\end{equation}
where the right side of the new identity is identical to
the right sides of the original identities, and:
\begin{equation}
 \frac{1}{2} [(-1)^k-(-1)^l] =
\begin{cases}
 (-1)^k & $\text{if $k+l$ is odd}$ \\
 0      & $\text{if $k+l$ is even}$ \\
\end{cases}
\end{equation}
where the right side of the new identity is zero.
For example taking identity \ref{resdbl2} with $a=0$, $b=c=1$ and $p=m$ and for the first identity
interchanging $k$ and $l$ and replacing $l$ by $n-k-l$,
and for the second identity additionally interchanging $k$ and $l$ again,
and taking linear combination (\ref{linsum1}) gives corollary 2.4 in \cite{merca1}:
\begin{equation}
 \dblsumeven (-1)^k q^{\textstyle\binom{n-k-l}{2}} \gaussian{m+k}{m}\gaussian{m+l}{m}\gaussian{m+1}{n-k-l}
 = \gaussian{m+n}{m}
\end{equation}
and taking identity \ref{resdbl3} with $a=b=c=1$ and $p=m+1$ and for the first identity replacing $l$ by $n-k-l$,
and for the second identity additionally interchanging $k$ and $l$, and taking linear combination
(\ref{linsum1}) gives corollary 3.4 in \cite{merca1}:
\begin{equation}
 \dblsumeven (-1)^k q^{{\textstyle\binom{k}{2}}+\textstyle\binom{l}{2}} \gaussian{m+1}{k}\gaussian{m+1}{l}\gaussian{m+n-k-l}{m} 
 = q^{\textstyle\binom{n}{2}} \gaussian{m+1}{n}
\end{equation}

\pdfbookmark[0]{References}{}

\end{document}